\documentclass[11pt]{article}
\usepackage[ansinew]{inputenc}
\usepackage[dvips]{graphicx}
\usepackage{amssymb,color}
\usepackage{color}
\usepackage{epic}

\def\BBox{\kern  -0.2cm\hbox{\vrule width 0.2cm height 0.2cm}}

\newtheorem{lemma}{Lemma}[section]
\newtheorem{theorem}{Theorem}[section]

\newtheorem{corollary}{Corollary}[section]

\title{
Advances on the Conjecture of Erd\H{o}s-S\'os for spiders}
\author{C. Balbuena$^1$, M. Guevara$^2$, J.R. Portillo$^3$, P. Reyes$^3$ \\[2 ex]
$^1$ {\small m.camino.balbuena@upc.edu}
\\ {\small\em Departament de Matem\`atica Aplicada III}\\
 {\small\em Universitat Polit\`ecnica de Catalunya}\\
$^2$ {\small mucuy-kak.guevara@ciencias.unam.mx.}
\\ {\small\em  Facultad de Ciencias}\\
 {\small\em Universidad Nacional Aut\'onoma de M\'exico, }\\
$^3$  {\small ~josera@us.es,~preyes@us.es}\\
{\small\em Departamento de Matem\'atica Aplicada I}\\
{\small\em Universidad de Sevilla}}

\date{}

\begin{document}

\maketitle
\begin{abstract}
Results:
\begin{itemize}
\item A hamiltonian graph $G$ verifying $e(G)>n(k-1)/2$  
contains any $k$-spider. 
\item If $G$ is a graph with average degree  $\bar{d} > k-1$, then every spider of size $k$ is contained in $G$ for $k\le 10$.
\item A $2$-connected graph with average degree  $\bar{d} > \ell_2+\ell_3+\ell_4$ contains every spider of $4$ legs $S_{1,\ell_2,\ell_3,\ell_4}$.
We claim also that the condition of $2$-connection is not needed, but the proof is very long and it is not included in this document.
\end{itemize}
\end{abstract}


\section{Introduction}
The Erd\H{o}s-S\'os conjecture \cite{E65} says that a graph $G$ on
$n$ vertices and number of edges $e(G)>n(k-1)/2$   contains all
trees of size $k$.

By $g(n,k)$ Erd\H{o}s and Gallai \cite{EG59} denoted  the maximum
number of edges of a graph $G$ on $n$ vertices containing no
cycles with more than $k$ edges. Moreover, these authors proved
that 
$g(n,k) \le \frac{1}{2}(n-1) k , \mbox{ for } 2\le k \le n.$ 
(\emph{Theorem 2.7})
Thus if $e(G)>  (n-1) k /2$ then $G$ contains a cycle with at
least $k+1$ edges.

Fan and Sun \cite{FL07} used Theorem 2.7 to note that every graph $G$ with
$e(G)>n(k-1)/2$ has a circumference of length at least $k$.
This is
clear because $e(G)>n(k-1)/2> (n-1)(k-1)/2$. Then they used this
observation to prove  that every graph with $e(G)>n(k-1)/2$
contains any $k$-spider of three legs. 
We will prove that   a hamiltonian graph $G$ with $e(G)>n(k-1)/2$  
contains any $k$-spider. We will prove also that a connected graph with $e(G)>n(k-1)/2$ contains all spider of four legs, with one leg of unity length.
\section{Results}

We need to introduce the following notation. Let $S_{\ell_1,\ldots,\ell_f}$ be a $k$-spider of $f$ legs of lengths $\ell_1,\dots,\ell_f$, i.e., $\ell_1+\dots+\ell_f=k$. Let $P_1, \dots, P_f$ be the $f$ legs of the $k$-spider such that $e(P_i)=\ell_i$, $i=1,2,\dots, f$. We may assume that $\ell_1\le \ell_2\le \cdots \le \ell_f$ and if the spider has $4$ legs ($f=4$) and $\ell_1=1$ then $\ell_2\ge 2$ because if $\ell_2=1$, then $S_{1,1,\ell_3,\ell_4}$ is a caterpillar, which is included in any graph $G$ with $e(G)>|V(G)|(k-1)/2$ where $k=\ell_3+\ell_4+2$ \cite{MP93}. Note that $\ell_1\le k/f$ and $\ell_2\le (k-\ell_1)/(f-1)$. (If $f=4$, $\ell_1\le k/4$ and $\ell_2\le (k-\ell_1)/3$).

In general, $\ell_i\le k-\sum\limits_{j=1}^{i-1}\ell_j/(f-i+1)=\sum\limits_{j=i}^{f}\ell_j/(f-i+1)$

Let $G$ be a graph with $e(G)>|V(G)|(k-1)/2$. Let $H$ be a minimal induced subgraph of $G$ such that $e(H)>|V(H)|(k-1)/2$. By the minimality, $H$ is connected and $deg_{H}(v)\ge k/2$ for every $v\in
V(H)$. If $H$ has a copy of $S_{\ell_1,\ell_2,\dots,\ell_f}$, so does $G$.

\begin{theorem}\label{hamilgen}
Let $G$ be a graph and $H$ a hamiltonian subgraph of $G$. Suppose that there exists a vertex $x_0\in V(H)$ such that $deg_H(x_0)\ge k$. Then $H$ contains (and so $G$) any $k$-spider.
\end{theorem}

\begin{proof}
Let $m=|V(H)|$ and $x_0 x_1 \cdots  x_{m-1} x_0$ a hamiltonian cycle of $H$. We will prove the theorem by induction on $k$.

For $k=2$, it is easy to check that theorem holds, because the only $2$-spider are $S_{1,1}$ and $S_2$ and they are isomorphic to   the path of length $2$, which is contained in the hamiltonian cycle. Moreover $S_{1,1}$ or $S_2$ can be taken with root $x_0$.
For $k=3$, the only  $3$-spiders are  $S_{1,1,1}$ that is contained in $H$ with root $x_0$ because $deg_H(x_0)\ge k=3$, and the isomorphic spiders $S_{1,2}$ and $S_3$ which clearly are contained in $H$ with $x_0$ as root.

Suppose that theorem is true for every $k'$ with $3< k'<k$, and let us show that the theorem is also valid for $k$. Let $S_{\ell_1,\ell_2,\dots,\ell_f}$ be a spider of $f$ legs and size $k$, i.e, $\ell_1+\ell_2+\dots+\ell_f=k$. Let $\alpha$ be the smallest index such that $x_{\alpha}\in N(x_0)$ with $\alpha \ge \ell_1+1$,
 that there exists because $deg_H(x_0)\ge k>\ell_1$.
 Let $H'\subset H$ be the subgraph induced by $V(H)\setminus \{x_1, \dots, x_{\alpha-1}\}$. Clearly $H'$ is hamiltonian because $C=x_0x_{\alpha}x_{\alpha+1}\cdots x_{n-1}x_0$ is a hamiltonian cycle of $H'$. Moreover   $deg_{H'}(x_0)\ge k-\ell_1$.
By the inductive hypothesis on $k$, $H'$ contains all the spiders with root in $x_0$ and size $k-\ell_1$. Particularly, $H'$ contains the spider $S_{\ell_2,\dots,\ell_f}$ whose legs are denoted by $P_2, \dots, P_f$. Then, the spider with root $x_0$ and legs $P_1=x_0,x_1,\dots, x_{\ell_1}$, $P_2,\dots, P_f$ is contained in $H$. Thus $S_{\ell_1,\ell_2,\dots,\ell_f}$ is contained in $G$, finishing the proof.
 \end{proof}

Observe that   $e(G)>|V(G)|(k-1)/2$ is equivalent to the requirement that the average degree $\overline{d}>k-1$ and so maximum degree of $G$, $\Delta(G)\ge \overline{d}>k-1$.

\begin{corollary}\label{hamilton}
Let  $G$ be a hamiltonian graph with average degree  $\bar{d} > k-1$. Then every spider of size $k$ is contained in $G$.
\end{corollary}

\begin{lemma}\label{ciclovalmax} 
Let $G$ be a 2-connected graph with average degree $\bar{d} > k-1$. Then
every vertex of degree at least $k$ lies on a cycle  $C_s$ of length $s\ge k$.
\end{lemma}

\begin{proof}
Since  $\bar{d}> k-1$ then $\Delta(G) \geq k$.
The results follows directly from Theorem 1.16 in  \cite{EG59}.
\end{proof}

\begin{corollary}\label{kmenorque11}
Let  $G$ be a graph with average degree  $\bar{d} > k-1$. Then every spider of size $k$ is contained in $G$ for $k\le 9$.
\end{corollary}

\begin{proof}
Every spider with legs of length at most 4 are contained in $G$ by Theorem 4.1 of \cite{FL07}. Moreover, every spider with three legs are contained in $G$ by Theorem 3.1 of \cite{FL07}. Therefore the remaining spiders are the comet $S_{1,1,1,1,5}$ and the caterpillar $S_{1,1,2,5}$ and therefore the result is valid by~\cite{MP93}.
\end{proof}


\begin{theorem}\label{mainbiconex}
Let $G$  be a 2-connected graph with average degree  $\bar{d} > k-1$. Then $G$ contains every $k$-spider $S_{1,\ell_2,\ell_3,\ell_4}\ (k=1+\ell_2+\ell_3+\ell_4)$.
\end{theorem}

\begin{proof}
We will suppose that $2\le \ell_2\le\ell_3\le\ell_3$ (otherwise $S_{1,1,\ell_3,\ell_4}$ is a caterpillar and is contained in $G$).
Let $x_0 \in V(G)$ be with $deg_G(x_0)=\Delta(G)\ge k$.
By Lemma~\ref{ciclovalmax}, we can take  a cycle $C_s$ of maximum length  $|C_s|=s\ge k$  such that $x_0\in V(C_s)$. Let $C_s=x_0  x_1 \cdots  x_{s-1}x_0$.
 If $N(x_0)\subset V(C_s)$, the subgraph $H$ of $G$ induced  by the vertices of $C_s$
  is clearly   hamiltonian   and has a vertex   $x_0$  of
Therefore by Theorem~\ref{hamilgen}, $H$ (and so $G$) contains all spiders of size $k$ and particularly $S_{1,\ell_2,\ell_3,\ell_4}$.

Hence assume that $N(x_0)\not\subset V(C_s)$. Therefore we can consider a path   $P= x_0 u_1\cdots u_\ell$  starting in $x_0$  of maximum length such that $V(C_s)\cap V(P)=\{x_0\}$. Two cases need to be distinguished according to $\ell\ge \ell_2$ or $\ell< \ell_2$.

\emph{Case 1:  $\ell\ge \ell_2$.}

If there exists $ y\notin V(C_s)\cup V(P)$ such that $y\in N(x_0)$, then  $S_{1,\ell_2,\ell_3,\ell_4}$ is contained in $G$ and its legs are
  $P_1=x_0y$, $P_2=x_0 u_1\cdots u_{\ell_2}$,
$P_3=x_0 x_1\cdots x_{\ell_3}$ and $P_4=x_0 x_{s-1}\cdots x_{s-\ell_4}$
(see Figure~\ref{fig:erdos-sos1}). Therefore, assume $N(x_0)\subset V(C_s)\cup V(P)$ and let us study the following subcases.

\begin{figure}
  \begin{minipage}[t]{.30\textwidth}
    \centering
    \includegraphics{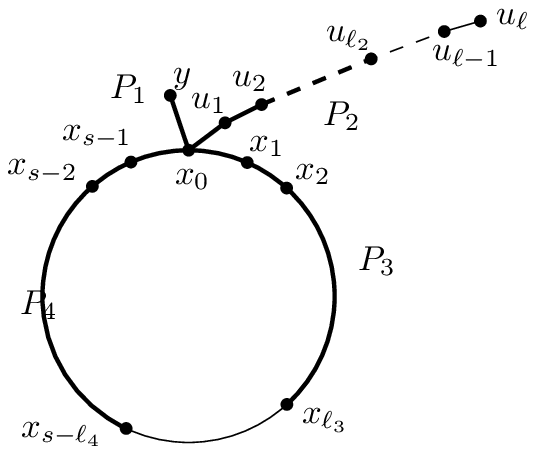}
    \caption{$\ell\ge \ell_2$ and $N(x_0)\not\subset V(C)\cup V(P)$.}
    \label{fig:erdos-sos1}
  \end{minipage}
   \hspace*{0.03\textwidth}
  \begin{minipage}[t]{.30\textwidth}
    \centering
    \includegraphics{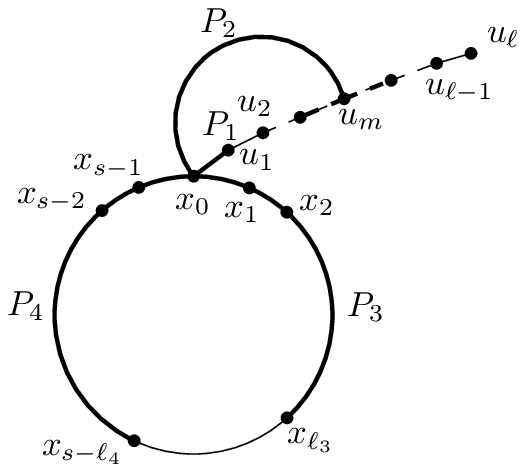}
    \caption{$\ell\ge \ell_2$, $u_m\in N(x_0)\cap V(P)$,   $2\geq m\geq \ell-\ell_2+1$
    or $\ell_2+1\le m$.}
    \label{fig:erdos-sos2}
  \end{minipage}
  \hspace*{0.03\textwidth}
  \begin{minipage}[t]{.30\textwidth}
    \centering
    \includegraphics{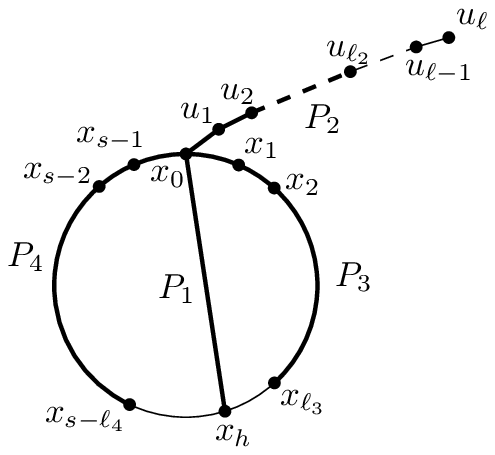}
    \caption{$\ell\ge \ell_2$, $N(x_0)\cap V(P)\subset \{u_{\ell-\ell_2+2},\dots,u_{\ell_2}\}$.}
    \label{fig:erdos-sos3}
  \end{minipage}
\end{figure}

\begin{enumerate}
\item[(a)] If $u_m\in N(x_0)$  with  $2\leq m\leq \ell-\ell_2+1$
, then $G$ contains the spider $S_{1,\ell_2,\ell_3,\ell_4}$,
of legs $P_1=x_0 u_1$, $P_2=x_0 u_m u_{m+1}\cdots u_{m+\ell_2-1}$, $P_3=x_0 x_1\cdots x_{\ell_3}$
and $P_4=x_0 x_{s-1}\cdots x_{s-\ell_4}$ (see Figure~\ref{fig:erdos-sos2}).

\item[(b)] If $u_m\in N(x_0)$  with $\ell_2+1\le m$, then $G$ contains the spider $S_{1,\ell_2,\ell_3,\ell_4}$  of legs
$P_1=x_0 u_1$, $P_2=x_0 u_m u_{m-1}\cdots u_{m-\ell_2+1}$, $P_3=x_0 x_1\cdots x_{\ell_3}$ and
$P_4=x_0 x_{s-1}\cdots x_{s-\ell_4}$ (see Figure~\ref{fig:erdos-sos2}).

\item[(c)] Otherwise we must distinguish between two different situations. If $\ell-\ell_2+2\leq \ell_2$ then $(N(x_0)-u_1)\cap V(P) \subseteq \{u_{\ell-\ell_2+2},\dots,u_{\ell_2}\}$, and $|N(x_0)\cap V(P)|\le \ell_2-(\ell-\ell_2+2)+1=2\ell_2-\ell-1\le 2\ell_2-\ell_2-1= \ell_2-1$. By the contrary, if $\ell-\ell_2+2> \ell_2$ then $N(x_0)\cap V(P)=\{ u_1\}$ and $|N(x_0)\cap V(P)|=1\le \ell_2-1$.
Therefore, in both cases,
$|N(x_0)\cap V(C_s)|\geq k-1- \ell_2+1=k- \ell_2=1+\ell_3+\ell_4$. Thus there must exist an edge $x_0x_h$ with $\ell_3<h<s-\ell_4$ so that $G$ contains the spider  $S_{1,\ell_2,\ell_3,\ell_4}$ of legs
  $P_1=x_0x_h$, $P_2=x_0 u_1\cdots u_{\ell_2}$,
$P_3=x_0 x_1\cdots x_{\ell_3}$ and $P_4=x_0 x_{s-1}\cdots x_{s-\ell_4}$ (see Figure~\ref{fig:erdos-sos3}).


\end{enumerate}

\emph{Case 2: $\ell< \ell_2$.}

Note that  $N(u_\ell)\subset V(C_s)\cup V(P)$ as $P$ has maximum length. Since $|N(u_\ell)\cap V(P)|\le \ell\le \ell_2-1$,
  $N(u_\ell)\not\subset V(P)$ because $ \ell_2\le(k-1)/3<k/2 \le deg(u_\ell)$.
Note also that since $C_s$ has maximum length, $  N(u_{\ell})\cap  \{x_1,\ldots, x_\ell\}=\emptyset$, otherwise the cycle
  $x_0u_1\cdots u_{\ell}x_ix_{i+1}\cdots x_{s-1}x_0$ would have a length greater than $s$ which is a contradiction.
 Similarly, $  N(u_{\ell}) \cap \{x_{s-\ell}, \ldots, x_{s-1}\}$ $=\emptyset$,
otherwise the cycle
$x_0u_1\cdots u_{\ell}x_i x_{i-1}\cdots x_{1}x_0$ would have a length greater than $s$ again a contradiction   (see Figure~\ref{fig:erdos-sos4}).

We may notice that if $x_j, x_{j+1}\in V(C_s)$ and $x_j\in N(u_\ell)$, then $ x_{j+1}\not\in N(u_\ell)$   because $C_s$ has maximum length.
Since $|\{x_{\ell+1},\dots,x_{\ell_2}\}|=|\{x_{s-\ell_2},\dots,x_{s-\ell-1}\} |=\ell_2-\ell $ then
$|N(u_\ell)\cap(  \{x_{\ell+1},\dots,x_{\ell_2}\}\cup\{x_{s-\ell_2},\dots,x_{s-\ell-1}\})|\le  2\left\lceil(\ell_2-\ell)/2\right\rceil$. Since $|N(u_\ell)\cap V(P)|\le \ell$, it follows that $|N(u_\ell)\cap  \{x_{\ell_2+1},\dots,x_{s-\ell_2-1}\}|\ge  k/2-2\left\lceil(\ell_2-\ell)/2\right\rceil-\ell\ge k/2-\ell_2-1>0$ (because $\ell\le \ell_2-1$).
Let $\alpha$ be the smallest index such that $x_\alpha\in N(u_\ell)\cap  \{x_{\ell_2+1},\dots,x_{s-\ell_2-1}\}$.
Since $k/2-\ell_2-1\le \left|N(u_\ell)\cap  \{x_{\ell_2+1},\dots,x_{s-\ell_2-1}\}\right|=\left|N(u_\ell)\cap  \{x_{\alpha},\dots,x_{s-\ell_2-1}\}\right|\le
\left\lceil(s-\ell_2-\alpha)/2\right\rceil\le (s-\ell_2-\alpha+1)/2 $ then $\alpha \le s-k+\ell_2+3 $.
But the inequality does not hold because it would mean that $2\left\lceil(\ell_2-\ell)/2\right\rceil=\ell_2-\ell+1$, so $u_{\ell}$ would be adjacent to $x_{s-\ell_2}$, and $\left\lceil(s-\ell_2-\alpha)/2\right\rceil=(s-\ell_2-\alpha+1)/2$ and $u_{\ell}$ would be adjacent to $x_{s-\ell_2-1}$ too, and that it is not possible because of the maximality of the cycle. So $\alpha < s-k+\ell_2+3 $.

\begin{figure}
  \begin{minipage}[t]{.30\textwidth}
    \centering
    \includegraphics{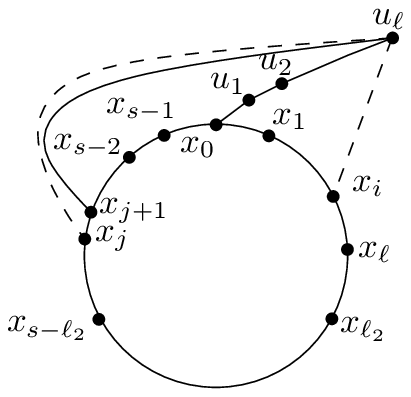}
    \caption{$\ell<\ell_2$.}
    \label{fig:erdos-sos4}
  \end{minipage}
   \hspace*{0.03\textwidth}
  \begin{minipage}[t]{.30\textwidth}
    \centering
    \includegraphics{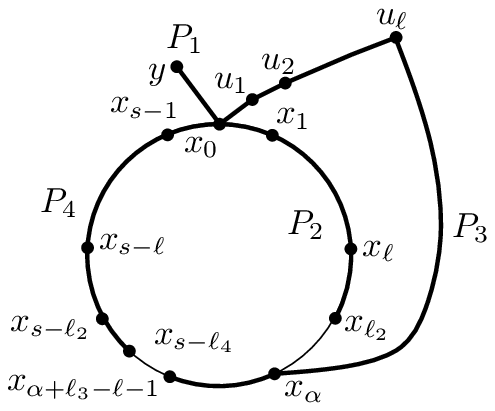}
    \caption{$\ell<\ell_2$, $N(x_0)\not\subset V(C)\cup V(P)$.}
    \label{fig:erdos-sos5}
  \end{minipage}
  \hspace*{0.03\textwidth}
  \begin{minipage}[t]{.30\textwidth}
    \centering
    \includegraphics{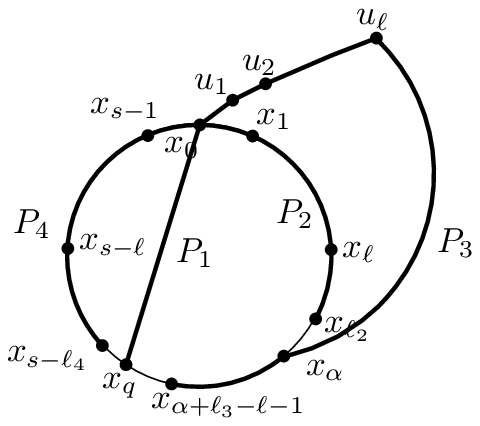}
    \caption{$\ell<\ell_2$, $N(x_0)\subset V(C)\cup V(P)$.}
    \label{fig:erdos-sos6}
  \end{minipage}
\end{figure}

If there exists $ y\notin V(C_s)\cup V(P)$ such that $y\in N(x_0)$, then  the spider $S_{1,\ell_2,\ell_3,\ell_4}$ with legs
  $P_1=x_0y$,  $P_2=x_0 x_1\cdots x_{\ell_2}$,
$P_3=x_0 u_1\cdots u_{\ell}x_\alpha\cdots x_{\alpha+\ell_3-\ell-1}$ and $P_4=x_0 x_{s-1}\cdots x_{s-\ell_4}$
 is contained in $G$ because $\alpha+\ell_3-\ell-1 <s-k+\ell_2+3+\ell_3-\ell-1=s-\ell_4+1-\ell\le s-\ell_4$ (see Figure \ref{fig:erdos-sos5}).
If $N(x_0)\subset V(C_s)\cup V(P)$, since $|N(x_0)\cap V(P)|\le \ell$, it follows that $|N(x_0)\cap V(C_s)|\ge k-\ell$. As the index set
$I=\{1,\ldots,\ell_2,\alpha,\ldots,   {\alpha+\ell_3-\ell-1},{s-\ell_4},\ldots,{s-1}\}$ has cardinality $k-\ell-1$, there must exist
  $x_q\in N(x_0)$ ($q\not\in I$). As
$\alpha+\ell_3-\ell-1<s-\ell_4$, the spider of legs
$P_1=x_0x_q$, $P_2=x_0  x_1\cdots x_{\ell_2}$, $P_3=x_0 u_1\cdots u_{\ell}x_\alpha\cdots x_{\alpha+\ell_3-\ell-1}$
and $P_4=x_0 x_{s-1}\cdots x_{s-\ell_4}$ is contained in $G$ (see Figure \ref{fig:erdos-sos6}).
 \end{proof}

\begin{theorem}\label{main}
If $G$ is a connected graph with average degree $\bar{d} > k-1$, then $G$ contains every
$k$-spider $S_{1,\ell_2,\ell_3,\ell_4}\ (k=1+\ell_2+\ell_3+\ell_4)$.
\end{theorem}

Proof is similar to the proof of Theorem~\ref{mainbiconex}, but is very long and it is not included here.

\begin{corollary}\label{k10}
Let  $G$ be a graph with average degree  $\bar{d} > k-1$. Then every spider of size $k$ is contained in $G$ for $k\le 10$.
\end{corollary}

\begin{proof}
Every spider with legs of length at most 4 are contained in $G$ by Theorem 4.1 of \cite{FL07}. Moreover, every spider with three legs are contained in $G$ by Theorem 3.1 of \cite{FL07}. For the caterpillars $S_{1,1,1,1,1,5}$,
$S_{1,1,1,1,6}$, $S_{1,1,1,2,5}$, $S_{1,1,1,7}$, $S_{1,1,2,6}$ or $S_{1,1,3,5}$  the result is valid by~\cite{MP93}. It remains the spider $S_{1,2,2,5}$, which can be applied the Theorem~\ref{main}.
\end{proof}


\end{document}